\documentclass[a4paper,10pt,reqno]{amsart}

\usepackage{amsmath}
\usepackage{amsthm}
\usepackage{amsfonts}
\usepackage{amssymb}
\usepackage{mathrsfs}
\usepackage[shortlabels]{enumitem}
\usepackage{graphicx}
\usepackage[dvipsnames]{xcolor}
\usepackage{tikz}

\usetikzlibrary{calc}
\usepackage{subfig}
\usepackage{subcaption}

\usetikzlibrary{matrix}

\usepackage{mathtools} % Francesco
\usepackage{courier} % Francesco
\usepackage[T1]{fontenc}  % Francesco

\usepackage{hyperref}
\usepackage{todonotes}

\newtheorem{theorem}{Theorem}[section]
\numberwithin{theorem}{section}

\newtheorem{lemma}[theorem]{Lemma}
\newtheorem{corollary}[theorem]{Corollary}
\theoremstyle{definition}
\newtheorem{definition}[theorem]{Definition}
\newtheorem{rem}[theorem]{Remark}

\numberwithin{equation}{section}

\newcommand{\N}{\mathbb{N}}

\newcommand{\R}{\mathbb{R}}

\newcommand\cH{\mathcal H}

\newcommand\ov\overline

\newcommand\eps\varepsilon
\renewcommand\epsilon\varepsilon
\renewcommand\rho\varrho
\newcommand\al\alpha
\newcommand\la\lambda

\newcommand\ds\displaystyle

\newcommand\p\partial

\newcommand{\beq}{\begin{equation}}
\newcommand{\eeq}{\end{equation}}
\newcommand{\be}{\begin{equation*}}
\newcommand{\ee}{\end{equation*}}
\newcommand{\bmat}{\begin{pmatrix}}
\newcommand{\emat}{\end{pmatrix}}

\newcounter{counter_a}

\newcommand{\diagdots}[3][-25]{%
  \rotatebox{#1}{\makebox[0pt]{\makebox[#2]{\xleaders\hbox{$\cdot$\hskip#3}\hfill\kern0pt}}}%
}

\author[F.~Ferraresso]{Francesco Ferraresso}
\address{Dipartimento di Informatica, Universit\`{a} degli Studi di Verona}
\email{francesco.ferraresso@univr.it}

\author[P.D.~Lamberti]{Pier Domenico Lamberti}
\address{Dipartimento di Tecnica e Gestione dei Sistemi Industriali, Universit\`{a} degli Studi di Padova}
\email{pierdomenico.lamberti@unipd.it}

\date{\today}

\thanks{}
\usepackage[update,prepend]{epstopdf}

\title[]{Steklov vs. Steklov: A Fourth-Order Affair Related to the Babu\v{s}ka Paradox}

\begin{document}

\maketitle
\vspace{-0.4cm}
{\centering \small\textit{To Professor Filippo Gazzola, on the occasion of his 60$^{th}$ birthday}.\par}

\begin{abstract} 
We discuss two fourth-order Steklov problems  and highlight a Babu\v{s}ka paradox appearing in their  approximations on convex domains via sequences of 
convex polygons. To do so,  we prove that the eigenvalues of one of the two problems  depend with continuity upon domain perturbation in the class of convex domains, extending a result known in the literature for the first eigenvalue. This is obtained by  examining in detail a nonlocal, second-order problem for harmonic functions introduced by Ferrero, Gazzola, and Weth. We further review how this result is connected to diverse variants of the classical Babu\v{s}ka paradox for the hinged plate and to a degeneration result by Maz'ya and Nazarov. 
\end{abstract}

%{\color{blue}

\section{Introduction}

Let $\Omega$ be a sufficiently regular  bounded  domain in  $\R^N$ (e.g., it is enough to assume that $\Omega$ is convex or of class $C^{1,1}$).
The Dirichlet Biharmonic Steklov (DBS) eigenvalue problem is defined as follows:
\begin{equation}\label{eq:DBS}
\begin{cases}
\Delta^2 u = 0, &\textup{in $\Omega$,} \\
u = 0, \quad &\textup{on $\p \Omega$,} \\
\Delta u = \la \, \p_\nu u, \quad &\textup{on $\p \Omega$,}
\end{cases}
\end{equation}
in the unknown $u \in H^2(\Omega) \cap H^1_0(\Omega)$ (the eigenfunction) and $\la \in \R$ (the eigenvalue). Note that in this paper $\nu$ denotes  the unit outer normal to $\partial\Omega$,  $H^m(\Omega)$ is the standard Sobolev space of functions in $L^2(\Omega)$ with  weak derivatives up to  order $m$ in $L^2(\Omega)$ and $H^m_0(\Omega)$ is the closure in $H^m(\Omega)$ of $C^{\infty}$-functions with compact support in $\Omega$. 

Problem \eqref{eq:DBS} is understood in weak sense; namely, we seek for a function $u \in H^2(\Omega) \cap H^1_0(\Omega)$ and a real number $\lambda$ verifying the variational equality
\begin{equation}\label{eq:DBSweak}
\int_\Omega \Delta u \Delta \varphi dx= \la \int_{\p \Omega} \p_\nu u \, \p_\nu \varphi \, d\sigma 
\end{equation}
for all $\varphi \in H^2(\Omega) \cap H^1_0(\Omega)$. 

It is well-known that problem \eqref{eq:DBSweak} has a divergent sequence of positive eigenvalues  of finite multiplicity that we denote by $\lambda_n(\Omega)$ and  that we order as 
$$
0<\lambda_1(\Omega)\le \dots  \le \lambda_n(\Omega)\le \dots 
$$
where each eigenvalue is repeated according to its multiplicity.

Problem \eqref{eq:DBS} is  classical and has attracted the attention of several mathematicians starting with Kuttler and Sigillito~\cite{kutsig68}. Among other papers, we quote \cite{antgaz, auch17,   bergazmit,  bucfergaz, bucgaz, buoso16, fergazwet,   FerreroLamb, gazpie, gazswe,  kutler72, kutler79, payne70}. One motivation for the study of this problem can be found by means of  the Fichera's Duality Principle which allows to prove that the first positive eigenvalue $\lambda_1(\Omega)$ of  \eqref{eq:DBS} is the best constant in the apriori estimate 
\begin{equation}\label{fich}
\lambda_1(\Omega)\| u\|_{L^2(\Omega)}^2\le \| u\|_{L^2(\partial\Omega)}^2
\end{equation}
for harmonic functions $u$ in $\Omega$, see \cite{fichera, kutsig68},  see also   \cite{fergazwet} for  more details. Moreover,  $\lambda_1(\Omega)$ plays a role in the study of a positivity preserving property of the solutions to certain fourth-order problems, see  \cite{bergazmit, gazswe}.

A variant of the (DBS), called Modified Dirichlet Biharmonic Steklov (MDBS) eigenvalue problem in \cite{FerreroLamb}, is instead defined as follows:
\begin{equation}\label{eq:MDBS}
\begin{cases}
\Delta^2 u = 0, &\textup{in $\Omega$,} \\
u = 0, \quad &\textup{on $\p \Omega$,} \\
\p^2_{\nu \nu} u = \mu \, \p_\nu u, \quad &\textup{on $\p \Omega$,}
\end{cases}
\end{equation}
in the unknown $u \in H^2(\Omega) \cap H^1_0(\Omega)$ (the eigenfunction) and $\mu \in \R$ (the eigenvalue). The corresponding weak formulation is: find  $u \in H^2(\Omega) \cap H^1_0(\Omega)$ verifying the variational equality
\begin{equation} \label{eq:MDBSweak}
\int_\Omega D^2 u : D^2 \varphi dx= \mu \int_{\p \Omega} \p_\nu u \, \p_\nu \varphi \, d\sigma 
\end{equation}
for all $\varphi \in H^2(\Omega) \cap H^1_0(\Omega)$. Here $D^2f$ denotes the Hessian matrix of a function $f$ and $A:B$ is the usual Frobenius product of two matrices $A,B$. 

Problem \eqref{eq:MDBS} also admits a divergent sequence of positive eigenvalues of finite multiplicity that we denote by $\mu_n(\Omega)$ and that we order as 
$$
0<\mu_1(\Omega)\le \dots  \le \mu_n(\Omega)\le \dots 
$$
taking into account their multiplicity.  

Up to our knowledge  problem   \eqref{eq:MDBS} was introduced in \cite{buoso16, FerreroLamb}. The main motivation for the study of this variant comes from the the Kirchhoff-Love model for a thin hinged plate since $\int_{\Omega}|D^2u|^2dx$ is one of the main terms in the elastic energy. 
Indeed, if $\Omega\subset \mathbb{R}^2$ represents the cross section of a thin hinged plate subject to an external load $f$ then the elastic energy is given by
$$
\int_{\Omega}\sigma (\Delta u)^2+(1-\sigma)|D^2u|^2 -fudx
$$
where $u\in H^2(\Omega)\cap H^1_0(\Omega)$  represents the deflection of the plate in vertical direction. 
Here $\sigma $ is the Poisson ratio of the plate that for physical reasons is often assumed to satisfy the condition $0\le \sigma \le 1/2$. Other values of $\sigma $ are also interesting in applications and, for $\Omega$ in $\mathbb{R}^N$, the typical assumption is  $-1/(N-1)<\sigma <1$, while  $\sigma =1$ represents a limiting case.

Thus the (DBS) problem corresponds to the limiting case $\sigma =1$ and the (MDBS) problem corresponds to the case $\sigma=0$. Although the other convex combinations of the two problems are clearly of interest, here we focus on the two extreme cases that exhibit a different behavior.  

\begin{figure}[htp]
\centering
\begin{tikzpicture}[scale=0.5]
% Definizione del raggio del cerchio 
\def\raggio{3}
% Calcola i vertici del poligono regolare inscritto
\foreach \angolo in {0,60,...,360} {
    \coordinate (P\angolo) at ({\raggio*cos(\angolo)},{\raggio*sin(\angolo)});
}
% Disegna la spezzata 
\foreach \angolo [evaluate=\angolo as \prossimo using \angolo+60] in {0,60,...,300} {
   \draw[black, very thick] (P\angolo) -- (P\prossimo);
}

\draw[thick, black] (0,0) circle(3cm);

\end{tikzpicture} \hfill
\begin{tikzpicture}[scale=0.5]
% Definizione del raggio del cerchio e della distanza per l'indentazione
\def\raggio{3}
% Calcola i vertici del poligono regolare circoscritto
\foreach \angolo in {0,30,...,360} {
    \coordinate (P\angolo) at ({\raggio*cos(\angolo)},{\raggio*sin(\angolo)});
}

% Disegna la spezzata con indentazioni
\foreach \angolo [evaluate=\angolo as \prossimo using \angolo+30] in {0,30,...,330} {
    \draw[black, very thick] (P\angolo) -- (P\prossimo);
}

\draw[thick, black] (0,0) circle(3cm);

\end{tikzpicture} \hfill
\begin{tikzpicture}[scale=0.5]
% Definizione del raggio del cerchio e della distanza per l'indentazione
\def\raggio{3}
% Calcola i vertici del poligono regolare circoscritto
\foreach \angolo in {0,20,...,360} {
    \coordinate (P\angolo) at ({\raggio*cos(\angolo)},{\raggio*sin(\angolo)});
}

% Disegna la spezzata con indentazioni
\foreach \angolo [evaluate=\angolo as \prossimo using \angolo+20] in {0,20,...,340} {
    \draw[black, very thick] (P\angolo) -- (P\prossimo);
}

\draw[thick, black] (0,0) circle(3cm);

\end{tikzpicture}
\caption{Approximation of the circle via convex polygons}
\label{fig:convexapp}
\end{figure}
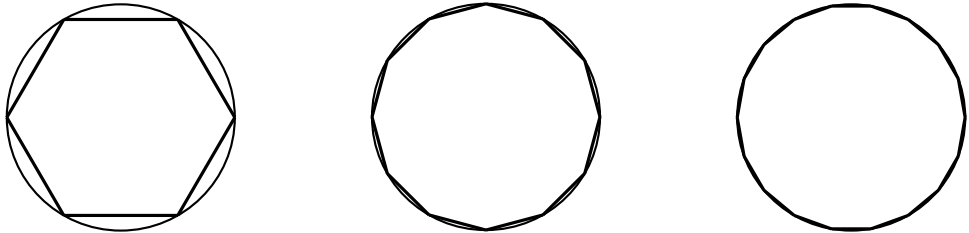

It has been pointed out in \cite[Theorems~3.7, 3.12]{H32} that for every $\sigma$ the (MDBS) problem is a limiting case of a family of biharmonic Steklov problems and that an analogous Neumann biharmonic Steklov problem is the limiting case of a second family of biharmonic Steklov problems. These two families of biharmonic Steklov problems have been introduced in \cite{H32} in order to characterize the trace spaces for functions in $H^2(\Omega)$ and to solve the Dirichlet biharmonic problem when $\Omega $ is a Lipschitz domain. 
 On the other hand, it is important  to observe that the several choices of Steklov  problems for $\Delta^2$ correspond to the several choices of different  Robin-type problems discussed in \cite{Buoken}. We refer to \cite{Buoken} for more details about such correspondence.

In this paper, we discuss a number of results concerning problems \eqref{eq:DBS} and \eqref{eq:MDBS}. One key point in this analysis is the hidden role of the curvature of $\partial\Omega$ in the boundary conditions of the two problems, a role that becomes evident in the appearance of a Babu\v{s}ka-type paradox. This paradox consists in the following three facts (see also Figure \ref{fig:convexapp}):
\begin{itemize}
\item[(i)] If $P$ is a convex polygon in the plane then $\lambda_n(P)=\mu_n(P)$ for all $n\in \N$.
\item[(ii)] If $D$ is the unit disk in the plane then $\lambda_n(D)=\mu_n (D)+1$.
\item[(iii)]  If $\{P_k\}_{k\in\N}$ is a sequence of regular polygons $P_k$ with $k$ sides inscribed or circumscribed to the unit disk $D$ then for every $n\in \N$
\begin{equation}\label{parastek} 
\lim_{k\to \infty}\mu_n(P_k)=\lim_{k\to \infty}\lambda_n(P_k)=\lambda_n(D)=\mu_n (D)+1\, .
\end{equation}
\end{itemize}

We note that,  in the case  where  $\Omega$ has a  regular boundary, the classical formula
\[
\Delta u|_{\p \Omega} = \p^2_{\nu \nu} u + \cH \p_\nu u + \Delta_{\p \Omega} u
\]
implies that the condition $\p^2_{\nu \nu} u = \mu \p_\nu u$ in the (MDBS) problem can be rewritten as
\[
\Delta u|_{\p \Omega} = (\cH + \mu) \p_\nu u .
\]
Here  and in the sequel $\cH $ denotes  the mean curvature of $\partial\Omega$ (the sum of the principal curvatures) and $\Delta_{\p \Omega}$  denotes the Laplace-Beltrami operator on $\partial\Omega$. In the case  of the unit disk $\cH =1$, hence  the shift $+1$ appears in \eqref{parastek} (see 
Theorem~\ref{nshift} for the $n$-dimensional unit ball).

The Babu\v{s}ka paradox  was first discovered in \cite{bab1961}  for the Poisson problem 
\begin{equation}\label{hingednav}
\begin{cases}
\Delta^2 u = f, &\textup{in $\Omega$,} \\
u = 0, \quad &\textup{on $\p \Omega$,} \\
\Delta u =0, \quad &\textup{on $\p \Omega$}
\end{cases}
\end{equation}
and its companion 
\begin{equation}\label{hinged}
\begin{cases}
\Delta^2 u = f, &\textup{in $\Omega$,} \\
u = 0, \quad &\textup{on $\p \Omega$,} \\
\p^2_{\nu \nu} u =0, \quad &\textup{on $\p \Omega$.}
\end{cases}
\end{equation}

In order to prove the Babu\v{s}ka Paradox for the Steklov eigenvalues we need  a stability result for  $\lambda_n(\Omega)$.
We note that the limit $\lim_{k\to \infty}\lambda_n(P_k)=\lambda_n(D)$ in \eqref{parastek} was first proved  in \cite{bucfergaz} for $n=1$ and circumscribed polygons. That result was later generalized in \cite{bucgaz} where it was proved that $\lambda_1(\Omega)$ depends with continuity on $\Omega$ with respect to the Hausdorff distance in the class of bounded convex domains. 

In this paper we prove that the continuity result of \cite{bucgaz}  holds for all eigenvalues $\lambda_n(\Omega)$, see Theorem~\ref{ncontinuity}. The proof of this stability result follows the main ideas of \cite[Thm.~5.1]{bucgaz}. In particular, one of the two involved semicontinuity results is obtained by exploiting a dual representation of the eigenvalues $\lambda_n(\Omega)$. Namely, in the spirit of the Fichera's Duality Principle, we define a sequence of dual eigenvalues $\delta_n(\Omega)$ and prove that $\lambda_n(\Omega)=\delta_n(\Omega)$ for all $n\in\N$. The equality $\lambda_1(\Omega)=\delta_1(\Omega)$ was already pointed out in Kuttler and Sigillito~\cite{kutsig68}  where it was formulated in terms of an equality between the minima of two Rayleigh quotients, see also \eqref{fich}.  Here we define  $\delta_n(\Omega)$ for all $n\in \N$ as the  eigenvalues of an operator in a Bergman  space, see Definition~\ref{ficheraweakeq} and Theorem~\ref{ficherahigh}. 

We observe that the instability of the eigenvalues $\mu_n(\Omega )$ upon variation of $\Omega$ in a family of domains with highly oscillating boundaries was proved in \cite{FerreroLamb} in the spirit of analogous results in   \cite{ArrLamb} for  the eigenvalue problem associated with  \eqref{hinged}. These results, which are reminiscent of the results in \cite{MazNaz} for indented polygons  (see Remark~\ref{mazrem}), are collected in the last of part of the paper. 

Finally, we mention that an extensive discussion of the two Steklov problems from the point of view of shape optimization has been recently carried out in \cite{sahakian}.

This paper is organized as follows. In Section~\ref{fichsec} we  prove the Fichera's duality Principle for all Steklov eigenvalues and apply it to prove the continuous dependence of $\lambda_n(\Omega)$ upon variation of $\Omega$. In Section~\ref{sec:classicalBab} we  discuss the new and the classical Babu\v{s}ka paradoxes. In Section~\ref{sec:oscillating} we survey a number of related results from   \cite{ArrLamb} and  \cite{FerreroLamb}.

\section{The Fichera's Principle for higher eigenvalues}\label{fichsec}

It is known since the work of Kuttler and Sigillito~\cite{kutsig68} that the first eigenvalue $\lambda_1(\Omega)$ of the biharmonic Steklov problem coincides with an eigenvalue $\delta_1(\Omega)$ of  a dual problem given by the minimization of  a suitable Rayleigh quotient, namely
$$
\delta_1(\Omega)=\inf_{\Delta h=0}\frac{\int_{\partial \Omega} h^2 d\sigma }{\int_{\Omega} h^2dx}.
$$ 

Kuttler and Sigillito claimed that  the equality $\lambda_1(\Omega)=\delta_1(\Omega)$  is a consequence of  the Fichera's duality Principle, a general principle  discussed by Fichera~\cite{fichera} in 1955 with several examples including the case of the biharmonic operator. 

Up to our knowledge, the first rigorous  proof of the equality  $\lambda_1(\Omega)=\delta_1(\Omega)$, as well as a clear definition of $\delta_1(\Omega)$, was  provided by Ferrero, Gazzola, and Weth~\cite{fergazwet} for domains of class $C^2$. Later,  the case of  Lipschitz domains satisfying the uniform outer ball condition or convex domains  was  considered  in Bucur, Ferrero, and Gazzola~\cite{bucfergaz} and in Bucur and Gazzola~\cite{bucgaz}. The equality  $\lambda_1(\Omega)=\delta_1(\Omega)$  is used in \cite{bucgaz} to prove the stability of $\lambda_1(\Omega)$ upon variation of $\Omega$ in the class of bounded convex domains. 

In this section we are going to  define the eigenvalues $\delta_n(\Omega)$  of higher order and to prove the equality $\lambda_n(\Omega)=\delta_n(\Omega)$ which will be used to extend the stability result of Bucur and Gazzola~\cite{bucgaz} to all eigenvalues $\lambda_n(\Omega)$.

In order to better clarify our plan, we note that  $\delta_1(\Omega)$ satisfies an eigenvalue problem that  in classical terms can be written as follows
\begin{equation}\label{ficheraclas}
\left\{\begin{array}{ll}
\Delta h=0,&\ {\rm in}\ \Omega,\\
h=\delta \, \p_\nu (\Delta^{\rm dir})^{-1}h,&\ {\rm on}\ \partial\Omega,
\end{array}\right.
\end{equation}
where $\delta$ is the eigenvalue, $h$ the eigenfunction,  and $\Delta^{\rm dir}$ is the usual Dirichlet Laplacian acting in the Hilbert space $L^2(\Omega)$.

The idea is to consider all eigenvalues $\delta_n(\Omega)$, $n\in \N$ of this problem. We note that proving the equality $\lambda_n(\Omega)=\delta_n(\Omega)$ would not be difficult if one could use the classical formulations of the problems and suitable regularity results. Indeed, the link between problem \eqref{ficheraclas} and problem \eqref{eq:DBS} is established by setting $h=\Delta u$ and using a rather simple argument.  Unfortunately, this method would require strong regularity assumptions on $\Omega$, for example one would require that $\Omega$ is at least of class $C^4$. Since we plan to require less regularity for $\Omega$, we need to use  weak formulations. This requires some care, in particular in the definition of the energy space associated with problem \eqref{ficheraclas}. 

From now on we shall always assume at least that $\Omega$ is a bounded domain with Lipschitz boundary satisfying the uniform outer ball condition, which means that there exists $r>0$ such that for any $x\in\partial \Omega$ there exists an open ball of radius $r$ such that $B\subset \mathbb{R}^N\setminus \Omega$ and $x\in \partial B$, cfr. \cite{adolfsson}.  Note that convex or $C^{1,1}$ domains are admissible. 

Following\footnote{In \cite{bucfergaz, bucgaz, fergazwet} the space $C^2_H(\bar \Omega)$ of harmonic functions of class $C^2$ up to the boundary is considered. In view of certain density arguments, we believe that $C^2_H(\bar \Omega)$ is suitable to study our problem in $C^{2,\alpha}$ domains. Having less regularity assumptions on $\Omega$ leads us to relax the assumption $u\in C^2(\bar\Omega)$ by replacing it with $u\in H^2(\Omega)$.} the idea of Ferrero, Gazzola, and Weth~\cite{bucfergaz}  we set
\begin{equation}
H^2_{\mathcal{H}}(\Omega )=\{u\in H^2(\Omega):\ \Delta u =0\},
\end{equation}
and we endow $H^2_{\mathcal{H}}(\Omega )$ with the norm defined by
$$
\| u\|_{H}=\left(\int_{\partial\Omega}u^2d\sigma\right)^{1/2}.
$$
Next, we consider the completion  $H(\Omega)$ of $(H^2_{\mathcal{H}}(\Omega ), \| \cdot \|_{H}  )$ and we note that $H(\Omega)$ is continuously  embedded in $L^2(\Omega)$.  Namely, $H(\Omega)$ is the space of all functions $u\in L^2(\Omega)$ such that there exists a Cauchy sequence $u_n\in H^2_{\mathcal{H}}(\Omega )$, $n\in \N$ (with respect to the norm $\|\cdot \|_{H}$) converging to $u$ in $L^2(\Omega)$: note that $u$ turns out to be harmonic (in the distributional sense) and $\| u\|_{H}=\lim_{n\to \infty }\| u_n\|_{H}$. 
This can be proved by  checking  that any Cauchy sequence in $H(\Omega)$ converges in $L^2(\Omega)$ and that, whenever $u_n$, $n\in \N$ is a Cauchy sequence with $u_n\to u$ in $L^2(\Omega)$ and $\| u_n\|_{H}\to 0$ as $n\to \infty$ then $u=0$ in $\Omega$. In fact, this is the case, since the apriori estimate  
 \begin{equation}\label{jerison}\| h\|_{H^{1/2}(\Omega)}\le c\, \| g\|_{L^2(\p \Omega)}\end{equation} for the Dirichlet problem
$$
\left\{
\begin{array}{ll}
\Delta h=0,& \ {\rm in}\ \Omega,\\
h=g,& \ {\rm on}\ \partial\Omega,\\
\end{array}
\right. 
$$
holds true, see \cite{jerison1, jerison2}.  Note that  given a function $h\in H(\Omega)$ it makes sense to consider its trace on $\partial \Omega$ which is uniquely identified as the limit in $L^2(\partial \Omega)$ of the trace of any Cauchy sequence $u_n\in H^2_{\mathcal{H}}(\Omega )$ converging to $h$ in $L^2(\Omega)$ and converging to some $g$ in $L^2(\partial \Omega)$: in this sense, $g$ is the trace of $h$ and we write that $h=g$ on $\partial \Omega$. 

We additionally note that the apriori estimate \eqref{jerison} allows to prove that $H(\Omega)$ is in fact continuously embedded into $H^{1/2}(\Omega)$,  where we set as usual
\[
H^{1/2}(\Omega) = \left\{ u \in L^2(\Omega) : \iint_{\Omega \times \Omega} \frac{|u(x) - u(y)|^2}{|x-y|^{N+1}} dxdy < \infty \right\};
\] 
we refer to \cite[Ch.\,VII, Thm.\,7.48]{Adams} for the definition of the fractional Sobolev spaces and their properties.  The space $H^2_{\mathcal{H}}(\Omega )$ plays the role of the energy space of problem \eqref{ficheraclas}.

 In order to formulate our problem in the weak sense, we also need a natural Hilbert space containing $H^2(\Omega)$. For this purpose, we consider the following Bergman  space\footnote{Note that this space coincides with  the classical Bergman space of harmonic functions in $L^2(\Omega)$ at least in sufficiently regular domains, see e.g. \cite[Lemma~8.8]{axler} where it is proved that harmonic polynomials are dense in the classical Bergman space in a ball.} by setting
 $$
 L^2_{\mathcal{H}}(\Omega)={\rm\ closure\  of}\ H(\Omega )\ {\rm in}\ L^2(\Omega).
 $$
Note that equivalently  $L^2_{\mathcal{H}}(\Omega)={\rm\ closure\  of}\ H^2_{\mathcal{H}}(\Omega )\ {\rm in}\ L^2(\Omega)$ by the definition of $H(\Omega )$.
We are ready to give the following definition
\begin{definition}\label{ficheraweakeq} We say that $\delta \in \R$ is a {\it dual (DBS) eigenvalue} if there exists $h\in H(\Omega)$, $h\ne 0$ such that 
\begin{equation}\label{ficheraweak}
\int_{\partial\Omega} h\psi d\sigma =\delta \int_{\Omega} h\psi dx,
\end{equation} 
for all $\psi \in H(\Omega)$. 
\end{definition}

 Up to our knowledge, the weak  formulation  \eqref{ficheraweak} appeared in this form  in \cite[\S~5]{fergazwet} for the first time.

Then we can prove the following theorem

\begin{theorem}\label{ficherahigh} The dual (DBS) eigenvalues are non-negative, have finite multiplicity, and they can be represented by means of a non-decreasing divergent sequence  $\delta_n(\Omega)$, $n\in\N$ (where  each eigenvalue is repeated according to its multiplicity), and  the  following Min-Max Principle holds: 
\begin{equation}\label{minmax}
\delta_n(\Omega)=\min_{\substack{E\subset H(\Omega)\\ {\rm dim}E=n}}\max_{u\in E}\frac{\int_{\partial \Omega} u^2d\sigma}{\int_{\Omega}u^2dx}.
\end{equation}
Moreover, there exists an orthonormal basis $h_n$, $n\in \N$, in $L^2_{\mathcal H}(\Omega)$ of eigenfunctions $h_n$ associated with the eigenvalues $\lambda_n$. 
\end{theorem}

\begin{proof}The proof is standard, here we follow the approach discussed in Davies~\cite[Ch.~4]{davies}. Namely, we consider the quadratic form 
$Q(u,v)=\int_{\partial\Omega}uvd\sigma$ defined for all $u,v\in H(\Omega)$. We note that $H(\Omega)$ is a dense subspace of the Hilbert space $L^2_{\mathcal{H}}(\Omega)$ and  the quadratic form is `closed' in the sense that the space $H(\Omega)$ is complete with the norm defined by $(Q(u,u)+\|u\|^2_{L^2(\Omega)})^{1/2}$. Accordingly, there exists a self-adjoint operator $T$ with domain $D(T)$ contained in $H(\Omega)$ such that the domain of its square root $T^{1/2}$ coincides with $H(\Omega)$ and  such that 
$Q(u,v)=\langle T^{1/2}u,T^{1/2}v\rangle_{L^2(\Omega)} $. In particular, $u\in D(T)$ if and only if $u\in D(T^{1/2})$ and $Tu\in D(T^{1/2})$ in which case 
$$
Q(u,v)=\langle T u,v\rangle_{L^2(\Omega)}
$$
for all $v \in D(T^{1/2})$. Moreover, since $H(\Omega) $ is embedded in $H^{1/2}(\Omega)$ and $H^{1/2}(\Omega)$ is compactly embedded into $L^2(\Omega)$, it follows that $H(\Omega) $  is compactly embedded into $L^2_{\mathcal{H}}(\Omega)$, hence $T$ has compact resolvent. Since the eigenvalues of $T$ are precisely the dual (DBS) eigenvalues, and the eigenfunctions are the same, the proof can be completed simply by using  the classical Min-Max Principle applied to the operator $T$. 
\end{proof}

The proof of Theorem~\ref{n-principle} below is an adaptation of the proof given in \cite[\S5]{bucfergaz} for the equality $\lambda_1(\Omega)=\delta_1(\Omega)$. However, a careful inspection of the crucial steps in the proof in  \cite[\S5]{bucfergaz} allows us to
avoid the use of the Min-Max Principle and to prove a one-to-one correspondence between eigenpairs of the two eigenvalue problems in a direct way as in  Lemma~\ref{equi} below.  Before stating and proving Lemma~\ref{equi}, we need the following technical result.

\begin{lemma}\label{approx}  Let $\Omega$ be  a  bounded Lipschitz domain  of $\R^N$ satisfying the uniform outer ball condition. Let $f\in L^2(\partial \Omega)$. Then there exists a sequence $u_n\in H^2_{\mathcal{H}}(\Omega )$, $n\in \mathbb{N}$ such that $u_n\to f$ in $L^2(\partial \Omega)$ as $n\to \infty$.
\end{lemma}

\begin{proof} We consider a sequence of functions $\tilde f_n\in H^{1/2}(\partial\Omega)$ such that $\tilde f_n\to f$ in $L^2(\partial \Omega)$ and $n\to \infty$. By the Trace Theorem, we can extend $\tilde f_n$ to the whole of $\Omega$ as a function in $H^1(\Omega)$. Possibly extending $\tilde f_n$ in a neighborhood of $\Omega$ as using a mollification argument, we can replace $\tilde f_n$ with a smoother  function, hence we can assume directly that $\tilde f_n$ is $H^2(\Omega)$.  Now consider a solution to the following Poisson problem
$$
\left\{ 
\begin{array}{ll}
\Delta v_n=-\Delta \tilde f_n,&  {\rm in }\ \Omega,\\
v_n=0,&  {\rm on }\ \partial\Omega .\\
\end{array}
\right.
$$
Since $\Omega$ is a bounded Lipschitz domain satisfying the uniform outer ball condition, it is well-known that the solution $v_n$ belongs to $H^2(\Omega)$, see \cite{adolfsson}. Setting $u_n=v_n+\tilde f_n$ we conclude.

\end{proof} 

\begin{rem}
If  $\Omega$ is assumed to be convex or of class $C^{1,1}$,  we can give an alternative proof of the previous lemma by using  Fourier series associated with the second order classical Steklov problem. Indeed, let $\varphi_j$, $j\in \mathbb{N}$ be an orthonormal basis in $L^2(\partial\Omega)$ of eigenfunctions of the  problem $\Delta \varphi_j = 0$ in $\Omega$ and $\p_\nu \varphi_j = \sigma_j \varphi_j$ on $\partial \Omega$, $\sigma_j$ being the classical Steklov eigenvalues. For domains of class $C^{1,1}$, the general results in Grisvard~\cite[Ch. 2]{grisvard85} imply that $\varphi_j\in H^2(\Omega)$; for domains that are convex, but not necessarily of class $C^{1,1}$, the $H^2$ regularity of Steklov eigenfunctions has been recently  established in  \cite{LambPro2025}. Then, writing  $f = \sum_{j=1}^\infty (f, \varphi_j)_{L^2(\p \Omega)}\, \varphi_j$ and  $u_n = \sum_{j=1}^n (g, \varphi_j)_{L^2(\p \Omega)} \varphi_j$ we get  $u_n \to g$ in $L^2(\p \Omega)$ as $n\to \infty$, and $u_n \in H^2_{\mathcal{H}}(\Omega)$. It seems plausible that the $H^2$ regularity of Steklov eigenfunctions holds on bounded Lipschitz domains satisfying the uniform outer ball condition. However, we have not been able to find this result in the literature and extending the arguments in   \cite{LambPro2025} to the general case would go beyond the scopes of the present paper.
\end{rem}

\begin{lemma}\label{equi}  Let $\Omega$ be  a   bounded Lipschitz domain  of $\R^N$ satisfying the uniform outer ball condition. The following statements hold:
\begin{itemize}
\item[(i)] Assume that $(\delta , h)\in \R_{+}\times H(\Omega)$ is an eigenpair of problem \eqref{ficheraweak}. Let $u\in H^2(\Omega)\cap H^1_0(\Omega)$ be the solution of the problem
\begin{equation}\label{equi1}
\left\{
\begin{array}{ll}
\Delta u=h,& \ {\rm in}\ \Omega,\\
u=0,& \ {\rm on}\ \p \Omega .\\
\end{array}
\right. 
\end{equation}
Then $(\delta , u)\in \R_{+}\times ( H^2(\Omega)\cap H^1_0(\Omega))$ is an eigenpair of problem \eqref{eq:DBSweak}.
\item[(ii)] Assume that $(\lambda , u)\in \R_{+}\times (H^2(\Omega)\cap H^1_0(\Omega))$ is an eigenpair of problem \eqref{eq:DBSweak}. Then the function $h:=\Delta u$ belongs to $H(\Omega)$ and $(\lambda , h)$ is an eigenpair of problem \eqref{ficheraweak}.
\end{itemize}
\end{lemma}

\begin{proof} We begin with statement (i). First of all we note that, under our regularity assumptions on $\Omega$,  problem \eqref{equi1} has indeed a solution in $H^2(\Omega)\cap H^1_0(\Omega)$, 
see   \cite{adolfsson}. Then  one can prove that 
\begin{equation}\label{equi2}
\int_{\partial\Omega} \Delta u \psi d\sigma =\delta \int_{\partial\Omega} \p_\nu u \,\psi d\sigma,
\end{equation} 
for all $\psi \in  H^2_{\mathcal{H}}(\Omega)$,  precisely as in \cite[\S5]{fergazwet}.  Thus, by the arbitrary choice of $\psi$ and Lemma~\ref{approx}, we deduce that 
\begin{equation}
\Delta u= \delta \, \p_\nu u 
\end{equation}
in the sense of the traces of $H(\Omega)$. Consider now a Cauchy sequence $v_n$ in $H^2_{\mathcal{H}}(\Omega)$ converging to 
$\Delta u$ in $L^2(\Omega)$ and to $\delta \p_\nu u$ in $L^2(\partial\Omega)$. By the Green's identity we get
$$
\int_{\Omega}v_n\Delta \psi dx=\int_{\partial \Omega} v_n \, \p_{\nu}\psi \, d\sigma
$$
for all $\psi\in H^2(\Omega)\cap H^1_0(\Omega)$. By passing to the limit as $n\to\infty $ in the previous equality we get 
$$
\int_{\Omega}\Delta u\Delta \psi dx=\delta \int_{\partial \Omega} \p_\nu u \, \p_\nu\psi \,d\sigma
$$
for all $\psi\in H^2(\Omega)\cap H^1_0(\Omega)$ hence $(u,\delta )$ is a (DBS) eigenpair. 

We now prove statement (ii).  Consider the function $g:= \lambda \,\p_\nu u \in L^2(\partial \Omega)$ (note that under our weak  assumptions on $\Omega$ we cannot assume that $g \in H^{1/2}(\partial \Omega)$). Using Lemma~\ref{approx}, take a sequence $g_n\in H^2_{\mathcal{H}}(\Omega)$  such that $g_n\to g$ in $L^2(\partial \Omega)$. By the estimate \eqref{jerison}, it follows that there exists $h\in H(\Omega)$ such that 
$g_n\to h$ in $L^2(\Omega)$   and $h=g$ on $\partial \Omega $ (in the sense of the space $H(\Omega)$). We now prove that $h=\Delta u$. In fact, for any $v\in H^2(\Omega)\cap H^1_0(\Omega)$ by the Green's formula we get
$$
\int_{\partial\Omega} g_n \, \p_\nu v \,d\sigma - \int_{\Omega} g_n \Delta v dx= \int_{\partial \Omega} \p_\nu g_n \, v \, d\sigma - \int_{\Omega} \Delta g_n v \,dx = 0
$$
hence
$$
\int_{\partial\Omega} g_n \p_\nu v \, d\sigma = \int_{\Omega} g_n \Delta v \,dx .
$$
By passing to the limit as $n \to \infty $ in the previous equality we get
$$
\lambda \int_{\partial \Omega} \p_\nu u \,\p_\nu v \, d\sigma = \int_{\Omega} h \Delta v \,dx.
$$
By the weak formulation of the Steklov problem we conclude that
$$
\int_{\Omega}h\Delta vdx=\int_{\Omega}\Delta u\Delta vdx.
$$
Since $v$ is arbitrary in $H^2(\Omega)\cap H^1_0(\Omega)$ we deduce that $h=\Delta u$. 

Finally, we consider an arbitrary function $v\in H^2_{\mathcal{H}}(\Omega)$ and again by the Green's formula we get
$$
\int_{\Omega} h v dx = \int_{\Omega}(\Delta u) v dx = \int_{\partial \Omega}\p_\nu u \, v \,d\sigma = \int_{\partial \Omega}\frac{g}{\lambda} v \,d\sigma = \int_{\partial \Omega}\frac{h}{\lambda}v \,d\sigma
$$
hence 
$$
\lambda \int_{\Omega}hv \,dx= \int_{\partial \Omega}hv \,d\sigma .
$$
By approximating any function in $H(\Omega)$ with functions of $H^2_{\mathcal{H}}(\Omega)$ we deduce that the previous equation is satisfied  for all $v\in H(\Omega)$ and the proof is complete.
\end{proof}

\begin{theorem}\label{n-principle}  If $\Omega $ is a bounded Lipschitz domain  of $\R^N$ satisfying the uniform outer ball condition  then  
$$\lambda_n(\Omega)=\delta_n(\Omega)
$$
for all $n\in \N$.
\end{theorem}

\begin{proof} The proof is an immediate consequence of the previous lemma combined with the observation that the multiplicity of the eigenvalues do not change in the one-to-one correspondence between the eigenpairs of the two eigenvalue problems. Indeed, if $u_1,u_2$ are two Steklov  eigenfunctions and they are orthogonal, that is  $\int_{\partial\Omega}u_1u_2d\sigma=0$ then the corresponding dual eigenfunctions $h_1=\Delta u_1$ $h_2=\Delta u_2$ are orthogonal in $L^2(\Omega)$ (and vice-versa).  
\end{proof}

The following theorem was proved for $n=1$ in  \cite[Thm.~5.1]{bucgaz} as a  generalization of the earlier result in \cite[Theorem~6]{bucfergaz} concerning the case of a sequence
of regular polygons circumscribed  about a  circle.

\begin{theorem}\label{ncontinuity}   The eigenvalues  $ \lambda_n(\Omega)$ depend with continuity  on bounded convex domains $\Omega\subset \R^N$ with respect to the Hausdorff distance  for all $n\in\N$.
\end{theorem}

\begin{proof} We adapt the proof of \cite[Thm.~5.1]{bucgaz} to the case of arbitrary $n\in \N$. Following the proof therein, we fix a bounded convex domain $\Omega$ and we  consider a sequence 
of bounded convex domains $\Omega_m$, $m\in\N$ converging to $\Omega$ with respect to the Hausdorff distance and we  prove that 
\begin{equation}\label{limsup}
\limsup_{m\to \infty }\delta_n(\Omega_m)\le \delta_n(\Omega)
\end{equation}
and 
\begin{equation}\label{liminf}
\lambda_n(\Omega)\le \liminf_{m\to \infty }\lambda_n(\Omega_m). 
\end{equation}
Since $\lambda_n(\Omega)=\delta_n(\Omega)$ for all $n\in \N$, the continuity of $\lambda_n(\Omega)$ will follow from \eqref{limsup} and \eqref{liminf}.

We begin with proving \eqref{limsup}. By possibly multiplying $\Omega_m$ by a factor $t_m\in ]0,1[$ such that $t_m\Omega_m\subset \Omega$ and $t_m\to 1$ as $m\to\infty$, and by rescaling the resulting eigenvalues, we can directly assume that $\Omega_m\subset \Omega$, see \cite{bucgaz}.
Then, for fixed $n\in \N$ we consider a set of dual eigenfunctions $h_1, \dots , h_n\in H(\Omega)$ satisfying the normalization conditions
$$
\int_{\partial \Omega} h_ih_jd\sigma =\delta_{ij}\delta_i(\Omega),\ \ {\rm and }\ \ \int_{ \Omega} h_ih_jdx =\delta_{ij}
$$
for all $i,j=1,\dots , n$, where $\delta_{ij}$ is the  Kronecker symbol. 

 We use the restrictions to $\Omega_k$ of the functions $h_j$, $j=1,\dots ,n$ in the min-max representation of $\delta_n(\Omega_k )$ provided by 
\eqref{minmax}. Note that such  restrictions  are linearly independent.  Without loss of generality, by the definition of the space $H(\Omega)$, we may assume for simplicity that $h_j\in H^2(\Omega)$ for all $j=1,\dots, n$. By the absolute continuity of the Lebesgue integral and by the same argument in the proof of  \cite[Thm.~5.1]{bucgaz} it follows that 
\begin{equation}\label{convbd}
\int_{\Omega_m}h_ih_j dx\to \int_{\Omega }h_ih_j dx,\ \ {\rm and}\ \ 
\int_{\partial\Omega_m}h_ih_j d\sigma \to \int_{\partial \Omega}h_ih_j d\sigma 
\end{equation}
as $m\to \infty$. We set $V_n=\{\sum_{j=1}^n \alpha_jh_j:\ \sum_{j=1}^n\alpha_j^2=1,\  \alpha_j\in \mathbb{R}, \ \forall\ j=1,\dots, n\} $.
Then by \eqref{minmax} and  \eqref{convbd}, we get
\begin{multline}\label{2trick}
\delta_n(\Omega_m)\le \max_{v\in V_n}\frac{\int_{\partial\Omega_m}v^2d\sigma}{\int_{\Omega_m}v^2dx}
\le \max_{v\in V_n}\frac{\int_{\partial\Omega }v^2d\sigma +\epsilon_1(n,m)}{\int_{\Omega}v^2dx-\int_{\Omega\setminus \Omega_m}v^2dx}\\
= \max_{v\in V_n}\frac{\int_{\partial\Omega }v^2d\sigma +\epsilon_1(n,m)}{1-\int_{\Omega\setminus \Omega_m}v^2dx}
\le 
 \max_{v\in V_n}\biggl( \int_{\partial\Omega }v^2d\sigma +\epsilon_1(n,m)\biggr) \biggl(1+2\int_{\Omega\setminus \Omega_m}v^2dx\biggr)\\
 \le 
 \max_{v\in V_n}\biggl( \int_{\partial\Omega }v^2d\sigma +\epsilon(n,m)\biggr) \biggl(1+\epsilon_2(n,m)\biggr)\\
 \le \delta_n(\Omega )+\epsilon_3(n,m,\delta_n(\Omega))
\end{multline}
provided $\epsilon_2(n,m)$ is sufficiently small,   where $\epsilon_1(n,m), \epsilon_2(n,m),  \epsilon_3(n,m,\delta_n(\Omega))$ are positive numbers (depending only the quantity in brackets) that converge  to zero as  $m\to \infty$. Note that in \eqref{2trick} we have used the elementary inequality $(1-x)^{-1}\le 1+2x$ if $0<2x<1$ with $x= \int_{\Omega\setminus \Omega_m}v^2dx\le \epsilon_2(n,m)/2$.
Inequality\eqref{2trick}  implies the validity of \eqref{limsup}. We refer to \cite{BurLa} for  analogous semicontinuity
estimates in a general  framework.

We now prove \eqref{liminf}.  For fixed $n\in N$ we consider a set of Steklov eigenfunctions $u_{1,m}, \dots , u_{n,m}\in H^2(\Omega_m)\cap H^1_0(\Omega_m)$ satisfying the normalization conditions
$$
\int_{ \Omega_m} \Delta u_{i,m}\Delta u_{j,m}dx =\delta_{ij}\lambda_i(\Omega_m),\ \ {\rm and }\ \ \int_{ \partial\Omega_m} \p_\nu u_{i,m} \, \p_\nu u_{j,m} \,d\sigma =\delta_{ij},
$$
for all $i,j=1,\dots , n$.
By the same argument in \cite[Thm.~5.1]{bucgaz}, we can prove that  there exist $u_1,\dots u_n \in H^2(\Omega)\cap H^1_0(\Omega)$ such that 
\begin{equation}\label{each}
\int_{\Omega}|\Delta u_i|^2dx \le \liminf_{m\to \infty } \int_{\Omega_m}|\Delta u_{i,m}|^2dx
\end{equation}
and 
\begin{equation}
\int_{ \partial\Omega} \p_\nu u_{i} \p_\nu u_{j}\, d\sigma = \lim_{m\to\infty }\int_{ \partial\Omega} \p_\nu u_{i,m} \, \p_\nu u_{j,m} \, d\sigma = \delta_{ij},
\end{equation}
for all $i,j=1,\dots n$.
In particular, $u_1, \dots , u_n$ are linearly independent.  Moreover, by the same argument used  to prove \eqref{each}, we can prove that  if $u=\sum_{j=1}^n\alpha_ju_j$ with $\alpha_j\in \R $ and $\sum_{j=1}^n\alpha_j^2=1$ then 
\begin{multline}
\int_{\Omega}|\Delta u|^2dx \le \liminf_{m\to \infty } \int_{\Omega_m}\biggl|\Delta \sum_{i=1}^n\alpha_iu_{i,m}\biggr|^2dx \\
\le  \liminf_{m\to \infty } \int_{\Omega_m}\sum_{i=1}^n \alpha_i^2(\Delta u_{i,m})^2  dx
\le  \liminf_{m\to \infty } \lambda_n(\Omega_m) .
\end{multline}

Finally,  by writing any  $u\in {\rm span}\{ u_1,\dots ,u_n \}$ with $\|\p_\nu u\|_{L^2(\partial\Omega)}=1$ as $u=\sum_{j=1}^n\alpha_ju_j$ with $\alpha_j\in \R $ and $\sum_{j=1}^n\alpha_j^2=1$, by the previous inequality and the Min-max Principle it follows that
\begin{equation}
\lambda_n(\Omega)\le \max_{\substack{u\in \langle u_1,\dots ,u_n\rangle \\ \| u_{\nu} \|_{L^2(\partial\Omega)} =1} }\int_{\Omega}|\Delta u|^2dx \le   \liminf_{m\to \infty } \lambda_n(\Omega_m) 
  \end{equation}
and the proof is complete. 
\end{proof}

\section{The new and the classical Babu\v{s}ka paradoxes} \label{sec:classicalBab}

The classical Babu\v{s}ka paradox \cite{bab1961}  for the hinged plate is an  example of   unexpected  asymptotic behavior for the solutions of a boundary value problem with respect to domain perturbations. Both the classical paradox and the paradox for the  fourth-order Steklov problem can be explained in the same way and are based on the following main result that can be found in \cite[Thm.~2.2.1]{grisvard92}

\begin{theorem}[Grisvard] \label{thm:Gris}
If $P$ is a polygon in $\R^2$ then 
\[
\int_P |D^2 u|^2dx = \int_P |\Delta u|^2dx
\]
for all $u \in H^2(P) \cap H^1_0(P)$.
\end{theorem}

Note that the convexity of $P$ is not required in the previous theorem but is relevant in other arguments used in the sequel. On the other hand, we warn the reader that assuming that $u\in H^2(P) $ is essential in Theorem~\ref{thm:Gris}  since it is well-known that in non-convex polygons $P$ a function may have the Laplacian in $L^2(P)$ without being in $H^2(P)$, something that does not happen on convex domains. Theorem~\ref{thm:Gris} allows to prove the following

\begin{corollary} \label{cor:Gris} If $P$ is a polygon in $\R^2$ then the two eigenvalue problems \eqref{eq:DBSweak} and \eqref{eq:MDBSweak} are the same. In particular,
$
\lambda_n(P)=\mu_n(P) 
$
for all $n\in \N$. 
\end{corollary}
\begin{proof}
By the polarization identity and Theorem~\ref{thm:Gris} it follows that 
\[
\int_{P} D^2 u : D^2 \varphi dx= \int_{P} \Delta u \Delta \varphi dx, \qquad \forall \varphi   \in H^2(P) \cap H^1_0(P)
\]
hence the two problems  \eqref{eq:DBSweak} are \eqref{eq:MDBSweak} coincide. The fact that  $P$ is bounded and  has a Lipschitz boundary allows to conclude that  the eigenvalues $\mu_n(P)$ exist hence $\lambda_n(P)$  exist as well and coincide with $\mu_n(P)$.  \end{proof}

\begin{rem}
Theorem \ref{thm:Gris} implies the stronger conclusion that the elastic energy of the plate with Poisson coefficient $\sigma$
\[
E(u) = \int_P \big( \sigma |\Delta u|^2 + (1-\sigma) |D^2 u|^2 \big) dx
\]
does not depend on $\sigma$ whenever $P$ is a polygon.
\end{rem}

By Theorems~\ref{ncontinuity} and Corollary~\ref{cor:Gris} we immediately deduce the following

\begin{corollary}[Babu\v{s}ka Paradox for (MDBS)] Let $\Omega$ be a a bounded convex domain in $\R^2$ such that $\lambda_n(\Omega)\ne \mu_n(\Omega)$ for some $n\in \mathbb{N}$. Let $\{P_k\}_{k}$ be a sequence of convex polygons converging to $\Omega$ with respect to the Hausdorff distance. Then 
$$
\lim_{k\to \infty }\mu_n(P_k)=\lim_{k\to\infty}\lambda_n(P_k)=\lambda_n(\Omega)\ne \mu_n(\Omega).
$$
\end{corollary}

Note that the condition $\lambda_n(\Omega)\ne \mu_n(\Omega)$ is satisfied for all $n\in {\mathbb{N}}$ in any $N$-dimensional ball. Indeed, we have the following

\begin{theorem}\label{nshift} Let $B$ be the  unit ball in $\R^N$. Then  
$$
\lambda_n(B)=2n-1+N-1\ \ \ {and}\ \ \ \mu_n(B)=2n-1
$$
for all $n\in\N$ and the corresponding eigenfunctions are the same and are of the form  $u_n(x)=|1-x|^2\psi_n(x)$ where $\psi_n$ is any homogeneous harmonic polynomial of degree $n-1$. 
\end{theorem}
\begin{proof}
 The eigenvalues $\lambda_n(B)$ and the corresponding eigenfunctions are provided by \cite[Thm.~1.3]{fergazwet}. The  formula for $\mu_n(B)$ follows by the formula for $\lambda_n(B)$ and by observing that 
$\Delta u|_{\p \Omega} = \p^2_{\nu \nu} u + (N-1) \p_\nu u
$ on $\partial B$ if $u=0$ on  $\partial B$. Obviously, $\lambda_n(B)$ and $\mu_n(B)$ have the same eigenfunctions.  
\end{proof}

\subsection{The classical paradox: the spectral version}In this section, we consider the spectral version of the classical  Babu\v{s}ka paradox in more detail.  Clearly, the spectral problem corresponding to \eqref{hingednav} is

\begin{equation}\label{hingednaveige}
\begin{cases}
\Delta^2 u = \xi u, \quad &\textup{in $\Omega$,} \\
u = 0, \quad & \textup{on $\p\Omega$,} \\
\Delta u = 0, \quad & \textup{on $\p\Omega$.}
\end{cases}
\end{equation}

Note that one natural way of formulating problem \eqref{hingednav} and \eqref{hingednaveige} is to consider them as problems associated with the square  of the Dirichlet Laplacian. Recall that the Dirichlet Laplacian $-\Delta_{\Omega} ^{\rm dir}$  is the selfadjoint operator in $L^2(\Omega)$ with domain 
$$
D(-\Delta_{\Omega} ^{\rm dir})= H(\Delta , \Omega )\cap H^1_0(\Omega)
$$
where $H(\Delta , \Omega )$ is the space of functions $u\in L^2(\Omega)$ with distributional Laplacian $\Delta u\in L^2(\Omega)$, and  it is defined by $-\Delta_{\Omega}^{\rm dir}u=-\Delta u$ for all $u\in H(\Delta , \Omega )\cap H^1_0(\Omega)$. 

The square   $(-\Delta_{\Omega} ^{\rm dir})^2$ of $-\Delta_{\Omega} ^{\rm dir}$  is the operator with domain given by 
$$
D((-\Delta_{\Omega} ^{\rm dir})^2)=\{u\in  H(\Delta , \Omega )\cap H^1_0(\Omega):\ \Delta u\in H(\Delta , \Omega )\cap H^1_0(\Omega)\}
$$
and it is defined by $(-\Delta_{\Omega} ^{\rm dir})^2u =\Delta (\Delta u)$. It turns out that given $f\in L^2(\Omega)$ a function $u\in D((-\Delta_{\Omega_n} ^{\rm dir})^2)$ is a solution to the problem 
$$
(-\Delta_{\Omega} ^{\rm dir})^2u=f
$$
if and only if $u\in H(\Delta , \Omega )\cap H^1_0(\Omega)$ and 
\begin{equation} \label{squareweak}
\int_\Omega \Delta u\, \Delta \varphi dx=  \int_{ \Omega}  f \, \varphi  dx
\end{equation}
for all $\varphi \in  H(\Delta , \Omega )\cap H^1_0(\Omega)$.  Equivalently, the solution $u$  to \eqref{squareweak} can be defined as the solution to the system 
\begin{equation}\label{eq:Dirsquared}
\begin{cases}
-\Delta u  = v, \quad &\textup{in $\Omega$,} \\
-\Delta v  = f, \quad &\textup{in $\Omega$,} \\
u = 0, \quad & \textup{on $\p \Omega$,} \\
v= 0, \quad & \textup{on $\p \Omega$.}
\end{cases}
\end{equation}
with $u,v\in  H(\Delta , \Omega )\cap H^1_0(\Omega)$. 

Thus, the eigenvalue problem \eqref{hingednaveige} can be understood as follows:  find $u\in H(\Delta , \Omega )\cap H^1_0(\Omega)$ and $\xi \in \R$ such that 
\begin{equation} \label{squareweakeige}
\int_\Omega \Delta u\, \Delta \varphi dx=  \xi \int_{ \Omega}  u \, \varphi  dx
\end{equation}
for all $\varphi \in  H(\Delta , \Omega )\cap H^1_0(\Omega)$.

If we denote the eigenvalues of the Dirichlet Laplacian by $\lambda_n^{\rm dir}(\Omega)$, it turns out that the eigenvalues $\xi$ of problem \eqref{squareweak} form a sequence
$$
0<\xi_1(\Omega)\le \dots  \le \xi_n(\Omega)\le \dots 
$$
given by
$$
\xi_n(\Omega)=(\lambda_n^{\rm dir}(\Omega))^2.
$$

The following well-known regularity result by Adolfsson~\cite{adolfsson} is fundamental in what follows since it allows to identify the space $H(\Delta , \Omega )\cap H^1_0(\Omega)$ in  \eqref{squareweakeige}.

\begin{theorem}\label{kadlec} If $\Omega $ is a bounded domain in $\R^N$ with Lipschitz boundary satisfying the uniform outer ball condition, then 
$$
H(\Delta , \Omega )\cap H^1_0(\Omega)=H^2( \Omega )\cap H^1_0(\Omega)
$$
and the corresponding norms are equivalent.
\end{theorem}
%In fact, we conclude that  if $\Omega$  is a bounded domain with Lipschitz boundary and satisfies the uniform outer ball condition then 
%of problem \ref{hingednaveige} can be written in the energy space $H^2(\Omega) \cap H^1_0(\Omega)$   \\

On the other hand, the spectral problem corresponding to \eqref{hinged} is
\begin{equation}\label{hingedeige}
\begin{cases}
\Delta^2 u = \eta u, \quad &\textup{in $\Omega$,} \\
u = 0, \quad & \textup{on $\p\Omega$,} \\
\p^2_{\nu\nu} u = 0, \quad & \textup{on $\p\Omega$.}
\end{cases}
\end{equation}
and the weak formulation reads
\begin{equation} \label{hingednaveigeweak}
\int_\Omega D^2 u : D^2 \varphi dx= \eta \int_{\Omega}  u \, \varphi  dx
\end{equation}
for all $\varphi \in H^2(\Omega) \cap H^1_0(\Omega)$, where $u \in H^2(\Omega) \cap H^1_0(\Omega)$ is the unknown eigenfunction and $\eta$ is the corresponding eigenvalue. If $\Omega$ is a bounded domain with  Lipschitz boundary, also this problem has a sequence of positive eigenvalues of finite multiplicity, represented by
$$
0<\eta_1(\Omega)\le \dots  \le \eta_n(\Omega)\le \dots .
$$

As for the Steklov problem, by Theorems~\ref{thm:Gris}, \ref{kadlec}  we immediately deduce the following

\begin{corollary} \label{cor:Grisbis} If $P$ is a convex polygon in $\R^2$ then the two eigenvalue problems \eqref{squareweakeige} are \eqref{hingednaveigeweak} are the same. In particular,
$
\eta_n(P) =\xi_n(P)= (\lambda_n^{\rm dir}(\Omega))^2
$
for all $n\in \N$. 
\end{corollary}

The stability of the eigenvalues of the Dirichlet Laplacian is well-known. Among various results, the following is in the same spirit of Theorem~\ref{ncontinuity}, see \cite[Thm.~2.3.17]{henrot}

\begin{theorem}\label{henrot} The maps $\Omega\mapsto \lambda_n^{\rm dir}(\Omega)$ defined  in the class of bounded convex domains $\Omega$ in $\R^N$  are continuous with respect to the Hausdorff distance. 
\end{theorem}

By Corollary~\ref{cor:Grisbis} and Theorem~\ref{henrot} we immediately deduce the following

\begin{corollary}[Classical Babu\v{s}ka Paradox - spectral version] Let $\Omega$ be a bounded convex domain in $\R^2$ such that $\eta_n(\Omega)\ne \xi_n(\Omega)$  for some $n\in \N$. Let $\{P_k\}_{k}$ be a sequence of convex polygons converging to $\Omega$ with respect to the Hausdorff distance. Then 
$$
\lim_{k\to \infty }\eta_n(P_k)=\lim_{k\to\infty}\xi_n(P_k)=\xi_n(\Omega)\ne \eta_n(\Omega).
$$
\end{corollary}
If $\p \Omega$ is sufficiently regular then  $\eta_n(\Omega)\ne \xi_n(\Omega)$  for all $n\in \N$. In fact, we have the following theorem which holds in any space dimension. Recall that $\cH$ denotes the mean curvature of the boundary (the sum of the principal curvatures).

\begin{theorem}
Let $\Omega$ be a bounded domain of class $C^2$  in $\R^N$ with $\inf_{x \in \p \Omega} \cH(x) > 0$. Then $\eta_n(\Omega) < \xi_n(\Omega)$ for all $n \in \N$. 
\end{theorem}
\begin{proof}
Define the Rayleigh quotients
\[
R_1(\varphi) = \frac{\int_\Omega |D^2 \varphi|^2  dx}{\int_\Omega \varphi^2 dx}, \qquad R_2(\varphi) = \frac{\int_\Omega (\Delta \varphi)^2 dx}{\int_\Omega \varphi^2 dx}
\]
for all $\varphi \in H^2(\Omega) \cap H^1_0(\Omega)$. By the Min-Max Principle
\begin{equation}\label{minmaxetaxi}
\eta_n(\Omega)=\min_{\substack{E\subset H^2(\Omega) \cap H^1_0(\Omega)\\ {\rm dim}E=n}}\max_{\varphi\in E}R_1(\varphi ),\ \ {\rm and}\ \ 
\xi_n(\Omega)=\min_{\substack{E\subset H^2(\Omega) \cap H^1_0(\Omega)\\ {\rm dim}E=n}}\max_{\varphi\in E} R_2(\varphi).
\end{equation}
Note that here we implicitly used Theorem~\eqref{kadlec} for the variational characterization of $\xi_n$. By the celebrated Reilly's  formula, we have 
\[
\int_\Omega |D^2 \varphi|^2 dx = \int_\Omega (\Delta \varphi)^2 dx- \int_{\p \Omega} \cH \bigg(\frac{\p \varphi}{\p \nu} \bigg)^2\, d\sigma
\]
for all $\varphi \in H^2(\Omega) \cap H^1_0(\Omega)$; thus,
\[
R_1(\varphi) \leq R_2(\varphi)
\]
for all $\varphi \in H^2(\Omega) \cap H^1_0(\Omega)$, which implies $\eta_n \leq \xi_n$ for all $n \in \N$, $n \geq 1$. \\
Assume for contradiction that $\eta_n = \xi_n$ for some $n$. Let $w_1, \dots, w_n$ be an orthonormal set of $n$  eigenfunctions $w_i$ associated with the eigenvalues $\xi_i$ of \eqref{hingednaveige}. Let $W = {\rm span}\{w_1, \dots, w_n\}$. By \eqref{minmaxetaxi} we conclude that 
\[
\xi_n = R_2(w_n) = \max_{\varphi \in W} R_2(\varphi).
\]
Using $W$ as a test linear space in the variational characterization  \eqref{minmaxetaxi}  of $\eta_n$ we obtain
\[
\xi_n=\eta_n \leq \max_{\varphi \in W} R_1(\varphi) = R_1(\hat{w}),
\]
where $\hat{w}$ is a maximizer for $R_1$ in $W$. There are now two cases:\\
Case 1. Assume that $\p_\nu \hat{w} = 0$ on $\p \Omega$. In this case 
\[
R_1(\hat{w}) = \frac{\int_\Omega |D^2 \hat{w}|^2 dx}{\int_\Omega \hat{w}^2 dx} = \frac{\int_\Omega (\Delta \hat{w})^2 dx}{\int_\Omega \hat{w}^2 dx} \geq	 \xi_n
\]
by the Reilly's formula, and the definition of $\hat{w}$; but then $\hat{w}$ is also a maximizer for $R_2$ on $W$, hence it is  an eigenfunction associated with $\xi_n$. Thus, $\hat w$ is an eigenfunction of the Dirichlet Laplacian in $\Omega$ with  $\hat w=\partial_{\nu}\hat w=0$ on $\partial\Omega$.  We conclude that $\hat w= 0$, a contradiction. Thus case 1 never occurs.\\
Case 2. Assume that  $\int_{\p \Omega} (\p_\nu \hat{w})^2\, d\sigma > 0$. Then 
\[
\begin{split}
\eta_n &\leq \frac{\int_\Omega |D^2 \hat{w}|^2 dx}{\int_\Omega \hat{w}^2dx} = \frac{\int_\Omega (\Delta \hat{w})^2 dx- \int_{\p \Omega} \cH (\p_\nu \hat{w})^2d\sigma}{\int_\Omega \hat{w}^2dx}\\
&\leq \frac{\int_\Omega (\Delta \hat{w})^2dx - \inf_{x \in \p \Omega}\cH(x) \int_{\p \Omega} (\p_\nu \hat{w})^2d\sigma}{\int_\Omega \hat{w}^2dx} \\
&< \frac{\int_\Omega (\Delta \hat{w})^2dx}{\int_\Omega \hat{w}^2dx} \leq \xi_n
\end{split}
\]
and therefore $\eta_n < \xi_n$ for all $n \geq 1$, as claimed.
\end{proof}

\section{Oscillating boundaries: trichotomy results} \label{sec:oscillating}
In this section we consider the couples of weak problems \eqref{eq:DBSweak}, \eqref{eq:MDBSweak} and \eqref{squareweakeige}, \eqref{hingednaveigeweak} on a family of bounded $C^{1,1}$ domains $\{\Omega_\eps\}_{\eps>0}$ converging in a suitable sense as $\eps \to 0^+$ to a bounded domain $\Omega$ of class $C^{1,1}$. In order to unify the presentation and to clarify the results, we will focus on domains in the following specific form.\\
We assume that $\Omega$ is of the type $\Omega=W\times (-1,0)$ where $W$ is a cuboid or a bounded domain in $\R^{N-1}$ of class $C^{1,1}$. We assume that the perturbed domain
$\Omega_\eps$ is given by
\begin{equation} \label{eq:Omega-eps-oscillating}
\Omega_\eps=\{(x',x_N):x'\in W\, ,
-1<x_N<g_\eps(x')\}
\end{equation}
where $g_\eps(x')=\eps^\alpha b(x'/\eps)$ for any $x'\in W$ and $b:\R^{N-1}\to [0,+\infty)$ is a $Y$-periodic function where $Y=\left(-\frac 12,\frac 12\right)^{N-1}$ is the unit cell in $\R^{N-1}$. We denote by $\Gamma_\eps$ and $\Gamma$ the sets
\begin{equation} \label{eq:Gamma-eps}
\Gamma_\eps:=\{(x',g_\eps(x')):x'\in W\} \quad \text{and} \quad
\Gamma:=W\times \{0\} \, .
\end{equation}
Let us also write $\Sigma_\eps = \p \Omega_\eps \setminus \Gamma_\eps$, for $\eps > 0$.

We recall now the following result regarding the behaviour of the eigenvalues of the (MDBS) Problem, see \cite[Thm.\,4.6]{FerreroLamb}. Note that the boundary condition $\partial_{\nu} u=0$ on $\Sigma_{\epsilon}$ was used for simplicity. 

\begin{theorem}[The trichotomy theorem - (MDBS) version] \label{thm:trichotomyMDBS} Let $\Omega_{\eps}$, $\eps \geq 0$ be as above, with $\Omega_0=\Omega$. Let $\mu_n(\eps )$, $n\in \N$,  be the eigenvalues of problems  
\begin{equation} \label{eq:Steklov-modificatoDIR}
\begin{cases}
\Delta^2 u=0, & \qquad \text{in } \Omega_{\eps} \, , \\
u=0, & \qquad \text{on } \partial\Omega_{\eps } \, , \\
\partial_{\nu}u=0, & \qquad \text{on } \Sigma_{\eps} \, , \\
\partial^2_{\nu\nu}u= \mu \, \partial_{\nu}u, & \qquad \text{on } \Gamma_{\eps
} \, ,
\end{cases}
\end{equation}
for  $\eps \geq 0$. Then the following statements hold:
\begin{itemize}
\item[(i)] If $\alpha >3/2$ then  $\mu_n(\eps )\to \mu_n(0
)$, as $\eps \to 0$. 
\item[(ii)] If $\alpha =3/2$ then
$\mu_n(\eps )\to  \mu_n(0 )+\gamma $ as $\eps \to 0$,
where 
\[\gamma=\int_{Y\times (-\infty , 0)  }|D^2V|^2dy,\]
and the function $V$ is $Y$-periodic in the variables $y'$ and  satisfies the following microscopic problem
\begin{equation}
\left\{
 \begin{array}{ll}
 \Delta^2V=0,\  & \hbox{\rm in } Y\times (-\infty, 0), \\
 V(y', 0)=b(y'),\ & \hbox{\rm on } Y, \\
 \frac{\partial^2 V}{\partial y_N^2}(y',0)=0,\  &\hbox{\rm on } Y.
  \end{array}
\right.
\end{equation} 
\item[(iii)] If
$1\le \alpha < 3/2$ and $b$ is non-constant then $\mu_n(\eps)\to \infty $, as $\eps
\to 0$.
\end{itemize}
\end{theorem}

Regarding the (DBS) eigenvalue problem, there are only partial results available. By means of the stability result (\cite[Theorem 3.2]{FerreroLamb}, see also Open Problem 5.1 \cite{FerreroLamb}) one can deduce the following.
\begin{theorem}[The stability theorem - (DBS) version] Let $\Omega_{\eps}$, $\eps \geq 0$ be as above, with $\Omega_0=\Omega$. Let $\lambda_n(\eps )$, $n\in \N$,  be the eigenvalues of problems  
\begin{equation} \label{eq:Steklov-DIR}
\begin{cases}
\Delta^2 u=0, & \qquad \text{in } \Omega_{\eps} \, , \\
u=0, & \qquad \text{on } \partial\Omega_{\eps } \, , \\
\partial_{\nu}u =0, & \qquad \text{on } \Sigma_{\eps} \, , \\
\Delta u= \la \, \partial_{\nu} u, & \qquad \text{on } \Gamma_{\eps
} \, ,
\end{cases}
\end{equation}
for  $\eps \geq 0$. Then, if $\alpha >3/2$, $\lambda_n(\eps )\to \lambda_n(0)$, as $\eps \to 0$. 
\end{theorem}

In the case of the classical hinged plate problem, a similar trichotomy was proved in \cite{ArrLamb} (see also \cite{FerLambpoly} for a trichotomy theorem in the context of polyharmonic operators with strong intermediate boundary conditions). We recall here the result. With $\Omega_\eps$, $\Omega$ as in \eqref{eq:Omega-eps-oscillating}, let $\eta_n(\eps)$ be the eigenvalues of the hinged plate problem in $\Omega_\eps$, defined by
\begin{equation}\label{eq:hinged}
\begin{cases}
\Delta^2 u_\eps = \eta_n(\eps) u_\eps, \quad &\textup{in $\Omega_\eps$}, \\
u_\eps = 0, \quad &\textup{on $\p \Omega_\eps$}, \\
\partial_{\nu}u=0,  \quad &\text{on } \Sigma_{\eps}, \\
\partial^2_{\nu\nu} u_\eps = 0, \quad &\textup{on $\Gamma_\eps$},
\end{cases}
\end{equation}
where $u_\eps \in H^2(\Omega_\eps) \cap H^1_0(\Omega_\eps)$ is the eigenfunction. We interpret this problem in weak sense: 
\[
\int_{\Omega_\eps} D^2 u_\eps : D^2 \varphi dx= \eta_n(\eps) \int_{\Omega_\eps} u_\eps \varphi dx,
\]
for all $\varphi \in H^2(\Omega_\eps) \cap H^1_0(\Omega_\eps)$ with $\p_\nu \varphi = 0$ on $\Sigma_\eps$. We further define the clamped plate eigenvalue problem by
\begin{equation}\label{eq:clampedtheta}
\begin{cases}
\Delta^2 u = \theta_n u, \quad &\textup{in $\Omega$}, \\
u = 0, \quad &\textup{on $\p \Omega$}, \\
\partial_{\nu}u=0,  \quad &\text{on } \p \Omega, \\
\end{cases}
\end{equation}
where $u \in H_0^2(\Omega)$ is the eigenfunction. We interpret this problem in weak sense: 
\[
\int_{\Omega } D^2 u : D^2 \varphi dx= \theta_n \int_{\Omega} u \varphi dx,
\]
for all $\varphi \in H_0^2(\Omega)$. Finally, let $\gamma$ and $V$ be as in Theorem \ref{thm:trichotomyMDBS}. We define the problem
\begin{equation}\label{eq:strange}
\begin{cases}
\Delta^2 u = \zeta_n u, \quad &\textup{in $\Omega$}, \\
u = 0, \quad &\textup{on $\p \Omega$}, \\
\partial_{\nu}u=0,  \quad &\text{on } \Sigma, \\
\Delta u + \gamma \partial_{\nu}u = 0, \quad &\textup{on $\Gamma$},
\end{cases}
\end{equation}
where $u \in H^2(\Omega) \cap H^1_0(\Omega)$ is the eigenfunction. Similar to the previous problems this has to be understood in weak sense. Then we have the following
\begin{theorem}[The trichotomy theorem - hinged plate version] \label{thm:hingedplate} Let $\Omega_{\eps}$, $\eps \geq 0$ be as above, with $\Omega_0=\Omega$. Let $\eta_n(\eps )$, $\theta_n$, $\zeta_n$, $n\in \N$,  be the eigenvalues of problems \eqref{eq:hinged}, \eqref{eq:clampedtheta}, \eqref{eq:strange} respectively. Then the following statements hold:
\begin{itemize}
\item[(i)] If $\alpha >3/2$ then  $\eta_n(\eps )\to \eta_n(0)$, as $\eps \to 0$. 
\item[(ii)] If $\alpha =3/2$ then $\eta_n(\eps )\to  \zeta_n $ as $\eps \to 0$.
\item[(iii)] If $0< \alpha < 3/2$ and   $b$ is non-constant  then $\eta_n(\eps)\to \theta_n $, as $\eps \to 0$.
\end{itemize}
\end{theorem}

\begin{figure}
\captionsetup[subfigure]{justification=raggedright}
    \centering
    \begin{minipage}[b]{0.45\textwidth}
        \centering
         \begin{tikzpicture}[scale=0.9]
			%\draw[step=1cm,gray,very thin] (-5,-5) grid (5,5);
			\draw[thick] (0,0) circle (3cm);
			\draw[very thick] ({3*cos(2.5)}, {-3*sin(2.5)}) -- (2.8, 0) -- ({3*cos(2.5)},{3*sin(2.5)}) -- ({3*cos(27.5)},{3*sin(27.5)}) -- ({2.5*cos(30)},{2.5*sin(30)}) -- ({3*cos(32.5)},{3*sin(32.5)}) ;
			\draw[very thick, blue] ({3*cos(32.5)},{3*sin(32.5)}) -- ({3*cos(85)},{3*sin(85)});
			\draw[very thick] ({3*cos(85)},{3*sin(85)}) -- (0,2.3) -- ({3*cos(95)},{3*sin(95)}) -- ({3*cos(117.5)},{3*sin(117.5)}) 			-- ({2.5*cos(120)},{2.5*sin(120)})--({3*cos(122.5)},{3*sin(122.5)}) -- ({3*cos(170)},{3*sin(170)}) -- (-2.2,0) -- ({-3*cos(5)},{-3*sin(5)}) -- ({-3*cos(27.5)},{-3*sin(27.5)}) -- ({-2.5*cos(30)},{-2.5*sin(30)}) -- ({-3*cos(32.5)},{-3*sin(32.5)}) -- ({-3*cos(87.5)},{-3*sin(87.5)}) -- (0,-2.5) -- ({-3*cos(92.5)},{-3*sin(92.5)}) -- ({-3*cos(117.5)},{-3*sin(117.5)}) -- ({-2.5*cos(120)},{-2.5*sin(120)})--({-3*cos(122.5)},{-3*sin(122.5)}) --  ({3*cos(2.5)}, {-3*sin(2.5)});
			\draw[thick, red] ({3*cos(85)},{3*sin(85)}) arc(85: 95 : 3);
			\draw ({2.75*cos(60)}, {2.75*sin(60)}) -- ({3.5*cos(60)},{3.5*sin(60)}) ;
			\draw ({3.7*cos(60)},{3.7*sin(60)}) node{{\tiny \color{blue} $z^{n}_j$}};
			\draw ({0},{3}) -- ({0},{3.5});
			\draw (0, 3.7) node{{\tiny \color{red} $\eta^{n}_j$}};
			\draw[thick, orange] (0, 2.31) -- (0,2.98);
			\draw (0,2.75) -- (-0.6, 2);
			\draw (-0.7, 1.8) node{{\tiny \color{orange} $h_j^{n}$}};
			%({3*sqrt(3)/2}, -3/2) -- ({3*sqrt(2)/2}, {3*sqrt(2)/2}) -- ({2.7*cos(47.5)}, {2.7*sin(47.5)}) -- ({3*cos(50)}, {3*sin(50)}) 				-- (-3/2, {3*sqrt(3)/2}) -- ({-3/2 + 0.1}, {3*sqrt(3)/2 - 0.4}) -- ({-3/2 - 0.3}, {3*sqrt(3)/2 - 0.2}) -- ({-3*sqrt(3)/2}, 					3/2) -- ({-3*sqrt(3)/2 + 0.3}, {3/2 -0.25}) --({-3*sqrt(3)/2 -0.15}, {3/2 -0.25}) -- ({-3*sqrt(2)/2}, {-3*sqrt(2)/2}) -- 					({-3*sqrt(2)/2 + 0.4}, {-3*sqrt(2)/2 + 0.25}) -- ({-3*sqrt(2)/2 + 0.3}, {-3*sqrt(2)/2 + 0.25 - 0.5}) -- (3/2, {-3*sqrt(3)/2}) 				-- ({3/2 + 0.05}, {-3*sqrt(3)/2 + 0.45}) -- ({3/2 + 0.3}, {-3*sqrt(3)/2 + 0.2}) -- ({3*cos(37)}, {-3*sin(37)}) -- 							({2.58*cos(33.5)}, {-2.58*sin(33.5)}) -- ({3*sqrt(3)/2}, -3/2); 
			\end{tikzpicture} 
            \caption{Indented polygon}\label{fignaz}
    	\end{minipage}\hfill
    \begin{minipage}[b]{0.45\textwidth}
        \centering
        \begin{tikzpicture}[scale=0.9]

% Definizione del raggio del cerchio e della distanza per l'indentazione
\def\raggio{1.94}
\def\indentazione{0.5}

% Calcola i vertici del poligono regolare circoscritto
\foreach \angolo in {0,10,...,360} {
    \coordinate (P\angolo) at ({\raggio*cos(\angolo)},{\raggio*sin(\angolo)});
}

% Disegna la spezzata con indentazioni
\foreach \angolo [evaluate=\angolo as \prossimo using \angolo+10] in {0,10,...,350} {
    % Calcola il punto medio del lato
    \coordinate (M\angolo) at ($0.5*(P\angolo) + 0.5*(P\prossimo)$);
    
    % Calcola il vertice interno spostandosi verso il centro
    \coordinate (I\angolo) at ($(M\angolo)!-\indentazione!(0,0)$);
    
    % Disegna il lato con indentazione
    \draw[black, very thick] (P\angolo) -- (I\angolo) -- (P\prossimo);
}

% Chiude il poligono collegando l'ultimo lato con indentazione
%\coordinate (M360) at ($0.5*(P360) + 0.5*(P0)$);
%\coordinate (I360) at ($(M360)!-\indentazione!(0,0)$);
%\draw[blue, thick] (P360) -- (I360) -- (P0);

% Disegna il cerchio circoscritto
%\draw[thick, gray] (0,0) circle(\raggio);
\draw[very thick, black] (0,0) circle(3cm);

\end{tikzpicture} % second figure itself
        \caption{Indented polygon with $z_j^n=0$}\label{fignaz2}
    \end{minipage}
\end{figure}

\begin{rem}\label{mazrem}
Case $(iii)$ in Theorem \ref{thm:hingedplate} is related to a classical degeneration theorem by Maz'ya and Nazarov  for the approximation of a circular plate via indented polygons, see \cite[Section 3]{MazNaz}. We briefly describe their result here. Let $\Omega = B(0,1)$ be  the unit disk in $\R^2$. We approximate $\Omega$ with a sequence of indented polygons $\Omega_n \subset \Omega$, see Figures~\ref{fignaz}, \ref{fignaz2}. Each indentation is of triangular shape with height $h_j^n < 1/2$, $h_j^n \to 0$ as $n \to + \infty$, for all $j =1, \dots, n$, and basis defining a chord of length $\eta_j^n$, $j=1, \dots, n$. The distance between two different indentations is $z_j^n$. The Maz'ya-Nazarov degeneration theorem states that if 
\begin{equation} \label{convcon}
\sup_n \max_{1 \leq j \leq n} \frac{|z_j^n|}{|\eta_j^n|} < \infty, \qquad \max_{1 \leq j \leq n} \frac{|\eta_j^n|}{(h_j^n)^{2/3}} \to 0,
\end{equation}
as $n \to + \infty$, then the sequence $u_n$ of solutions to the Navier plate problem \eqref{hingednav} with smooth datum $f$ in the indented polygon $\Omega_n$ converges in $H^2(K)$, $K$ any subdomain verifying $\ov{K} \subset \Omega$, to the solution $u$ of
\begin{equation}\label{eq:clamped}
\begin{cases}
\Delta^2 u = f, \quad &\textup{in $\Omega$,} \\
u = 0, \quad & \textup{on $\p\Omega$,} \\
\p_{\nu} u = 0, \quad & \textup{on $\p\Omega$.}
\end{cases}
\end{equation}
To fix the ideas, assume that $|\eta_j^n| = 1/n = \eps$ and $h_j^n = \eps^\alpha$. Then the second condition in \eqref{convcon} can be read as
\[
\max_{1 \leq j \leq n} \frac{\eps}{\eps^{2\alpha/3}} \to 0 \quad \Rightarrow \quad \alpha < 3/2
\]
and therefore we recover the critical exponent $\frac{3}{2}$ of Theorem~\ref{thm:hingedplate}.
\end{rem}

\begin{rem} The stability results stated  in the above theorems for $\alpha >3/2$ can be deduced from  general stability results concerning  domains $\Omega_{\epsilon}$ and $\Omega$ that can be locally represented as the subgraphs of suitable functions $g_{\epsilon}$ and $g$, more general than those considered in the specific case \eqref{eq:Omega-eps-oscillating}. Namely, one can assume that there exists a finite collection of open cuboids $V$ (that play the role of local charts) such that for every fixed cuboid $V$ one has 
\[
\begin{split}
&\Omega_\eps \cap V =\{(x',x_N):x'\in W,\, -1<x_N<g_\eps(x')\},\\ 
&\Omega \cap V = \{(x',x_N):x'\in W,\, -1<x_N<g(x')\},
\end{split}
\]
up to a local change of coordinates depending on $V$. Using the terminology  of \cite{ArrLamb,FerreroLamb}, such open sets belong to the same `atlas class'. Proving the stability results in this more general framework, requires a number of technicalities, including suitable partitions of unities. We refer the interested reader to \cite{ArrLamb,FerreroLamb} for more details.
\end{rem}

\section*{Acknowledgments}
The authors are deeply  indebted  to  Prof. Alberto Ferrero for very helpful  discussions.  The authors  are members of the ``Gruppo Nazionale per l’Analisi Matematica, la Probabilit\`a e le loro Applicazioni'' (GNAMPA) of the ``Istituto Nazionale di Alta Matematica'' (INdAM). The first named author FF acknowledges financial support under the National Recovery and Resilience Plan (NRRP), Mission 4 Component 2 Investment 1.5 - Call for tender No.3277 published on December 30, 2021 by the Italian Ministry of University and Research (MUR) funded by the European Union – Next Generation EU, project Code ECS0000038 – eINS Ecosystem of Innovation for Next Generation Sardinia – CUP J83C21000320007 - Grant Assignment Decree No. 1056 adopted on June 23, 2022 by the MUR.  The second named author PDL  acknowledges the support  from the project 
 ``Perturbation problems and asymptotics for elliptic differential equations: variational and potential theoretic methods'' funded by the European Union - Next Generation EU and by MUR Progetti di Ricerca di Rilevante Interesse Nazionale (PRIN) Bando 2022 grant 2022SENJZ3.

\bibliographystyle{abbrv}
\bibliography{BabbibPIER}

\end{document}